\pdfoutput=1
\documentclass[11pt]{amsart}

\usepackage{amsmath,amssymb,amsthm,mathtools}
\usepackage{bm}
\usepackage{graphicx}
\usepackage{microtype}
\usepackage[hidelinks]{hyperref}
\hypersetup{
  pdftitle={Universal Exponent-Two Degree Laws in Range-Renewal Networks},
  pdfauthor={Jiansheng Xie and Yechi Zhou},
  pdfkeywords={random graphs, stochastic processes, infinite occupancy schemes, regular variation, degree distributions}
}

\newcommand{\Pnum}{\mathbb{P}}
\newcommand{\Enum}{\mathbb{E}}
\newcommand{\RV}{\mathrm{RV}}
\newtheorem{theorem}{Theorem}
\newtheorem{lemma}{Lemma}
\newtheorem{proposition}{Proposition}
\theoremstyle{remark}
\newtheorem{remark}{Remark}

\title[Exponent-Two Laws in Range-Renewal Networks]{Universal Exponent-Two Degree Laws in Range-Renewal Networks}

\author{Jiansheng Xie}
\address{School of Mathematical Sciences, Fudan University, Shanghai 200433, China}
\email{jsxie@fudan.edu.cn}

\author{Yechi Zhou}
\address{School of Mathematical Sciences, Fudan University, Shanghai 200433, China}
\email{25210180114@m.fudan.edu.cn}

\keywords{random graphs, stochastic processes, infinite occupancy schemes, regular variation, degree distributions}

\begin{document}

\begin{abstract}
Let an infinite sequence of independent and identically distributed random variables over a countable alphabet generate a graph by joining consecutive symbols and suppressing repeated edges. We determine the exact tail and local asymptotics of the limiting degree distributions of this range-renewal graph. If the ordered sampling probabilities satisfy $\pi_k\in\RV_{-1/\gamma}$ with $0<\gamma<1$, then the directed and undirected degree tails are asymptotic to $\pi_k^\gamma$ and $2^\gamma\pi_k^\gamma$, respectively, while the corresponding local masses are asymptotic to $\pi_k^\gamma/k$ and $2^\gamma\pi_k^\gamma/k$. Consequently, both limiting laws are regularly varying with index $-2$, independent of $\gamma$; forgetting edge orientation affects only the leading amplitude. The proof combines infinite-occupancy estimates for discovery times and residual unseen mass with a conditional geometric representation of inter-discovery gaps. Uniform integrability yields the tail asymptotics, whereas a geometric-smoothing argument obtains the local masses without differentiating a regularly varying tail. We also prove that deleting self-loops leaves the limiting laws unchanged. Finite-sample simulations for normalized Zipf frequencies illustrate the asymptotic result.
\end{abstract}

\maketitle

\section{Introduction}\label{sec:intro}

Broad degree distributions are central to network theory and have been studied from several viewpoints \cite{AlbertBarabasi2002,Newman2003,NewmanStrogatzWatts2001}. Well-known mechanisms include preferential attachment and hidden variables \cite{BarabasiAlbert1999,Caldarelli2002,BogunaPastor2003}. Random sequences provide a different source of random graphs: observed symbols form the vertices, and local relations in the sequence generate edges. Examples include word-adjacency networks \cite{FerrerSole2001,MasucciRodgers2006,MasucciRodgers2009}, adjacency multigraphs generated from independently sampled alphabets \cite{BedogneRodgers2008}, and graph representations of time series \cite{Campanharo2011,Zou2019}. Such constructions raise the following question: given a heavy-tailed sampling law, how does the infinite-occupancy structure of the sequence determine the asymptotic degree law after repeated transitions have been merged?

We study the range-renewal construction of Chen, Xie, and Ying \cite{CXY2013,CXY2022}. Let $\xi_1,\xi_2,\ldots$ be independent samples from a probability distribution $\pi=(\pi_i)_{i\ge1}$ on a countable alphabet. Every observed symbol is a vertex, each consecutive pair $(\xi_i,\xi_{i+1})$ generates a directed edge, and repeated edges are suppressed. Forgetting orientation gives an undirected graph. Although the symbols are independent, the edge sequence is $1$-dependent because adjacent edges share an endpoint. Earlier work established almost-sure limits for the directed and undirected degree frequencies when $\pi_k\in\RV_{-1/\gamma}$, $0<\gamma<1$, and conjectured that the local limiting masses exhibit a power law with exponent $2$ for every $\gamma$ \cite{CXY2022}.

We prove this conjecture and obtain exact first-order asymptotics for four quantities: the directed and undirected limiting tails and their corresponding local masses. Both local laws are regularly varying with index $-2$. The parameter $\gamma$ remains in a slowly varying factor and in the orientation-dependent amplitude $2^\gamma$, but not in the degree exponent. Thus the conclusion is not merely an exponent identification: it also determines the leading constants and the ratio between the directed and undirected laws. We further show that deleting self-loops changes neither limiting distribution.

The local result is not a formal consequence of the tail result. A regularly varying tail need not have an asymptotically regular one-step decrement without additional information. The proof therefore separates the two tasks. Infinite-occupancy estimates for the discovery time $N_k$ and the residual unseen mass $M_k$ first yield the limiting tails through uniform integrability. Conditional on the history at the $k$th discovery, the next-discovery gap is geometric with parameter $M_k$. A geometric-smoothing lemma then resolves the one-step scale and yields the local masses directly. This argument avoids differentiating a tail asymptotic and may be useful in other random structures governed by discovery processes.

Range-renewal graphs are related to, but different from, edge-exchangeable and rank-one graph models \cite{CraneDempsey2018,Janson2018,CaronFox2017}. Janson's construction uses independent endpoint pairs and deletes self-loops, whereas consecutive pairs in range renewal overlap. For power-law activities, both models exhibit degree tails of order $k^{-1}$, but the present problem asks for sharp limiting tail and local-mass asymptotics in the overlapping-edge setting. Here the mechanism is the interaction between infinite occupancy and the suppression of repeated transitions.

The limits studied below are iterated: first the sequence length $n$ tends to infinity for fixed degree $k$, and then the limiting law is analysed as $k\to\infty$. Section~\ref{sec:model} defines the model and recalls the limiting degree representation. Section~\ref{sec:results} states the exact asymptotics, Sections~\ref{sec:representation} and \ref{sec:proof} develop the discovery representation and proof, and Section~\ref{sec:numerics} gives a finite-sample illustration for normalized Zipf frequencies. We conclude with methodological consequences and open problems.

\section{Sequence-generated network model}\label{sec:model}

Let $\pi=(\pi_i)_{i\ge1}$ be a nonincreasing probability sequence with $\pi_i>0$ and $\sum_i\pi_i=1$, and let $\xi_1,\xi_2,\ldots$ be independent and identically distributed with law $\pi$. After $n$ samples, write the observed vertex set and its size as
\begin{equation}
V_n:=\{\xi_1,\ldots,\xi_n\},\qquad R_n:=\#V_n.
\end{equation}
For ordered symbols $x,y$, define the number of observed consecutive transitions
\begin{equation}
W_n(x,y):=\sum_{i=1}^{n-1}\mathbf{1}_{\{\xi_i=x,\xi_{i+1}=y\}}.
\end{equation}
The directed out-neighbourhood and out-degree of $x$ are
\begin{equation}
V_n^+(x):=\{y:W_n(x,y)\ge1\},\quad D_n^+(x):=\#V_n^+(x),
\end{equation}
and the number of vertices with out-degree $k$ is
\begin{equation}
\vec R_{n,k}:=\#\{x\in V_{n-1}:D_n^+(x)=k\}.
\end{equation}
For the undirected projection, put
\begin{align}
\bar V_n(x)&:=\{y:W_n(x,y)+W_n(y,x)\ge1\},\\
\bar D_n(x)&:=\#\bar V_n(x),
\end{align}
and
\begin{equation}
\bar R_{n,k}:=\#\{x\in V_n:\bar D_n(x)=k\}.
\end{equation}
Repeated transitions are suppressed in both networks; a self-loop contributes one neighbour, as shown in Fig.~\ref{fig:model}. We use this retained-loop convention in the main derivation; the loop-deleted version is treated in Appendix~\ref{app:loop_deletion}, where we show that deleting self-loops leaves the limiting degree distributions unchanged.

\begin{figure}[t]
\centering
\includegraphics[width=\linewidth]{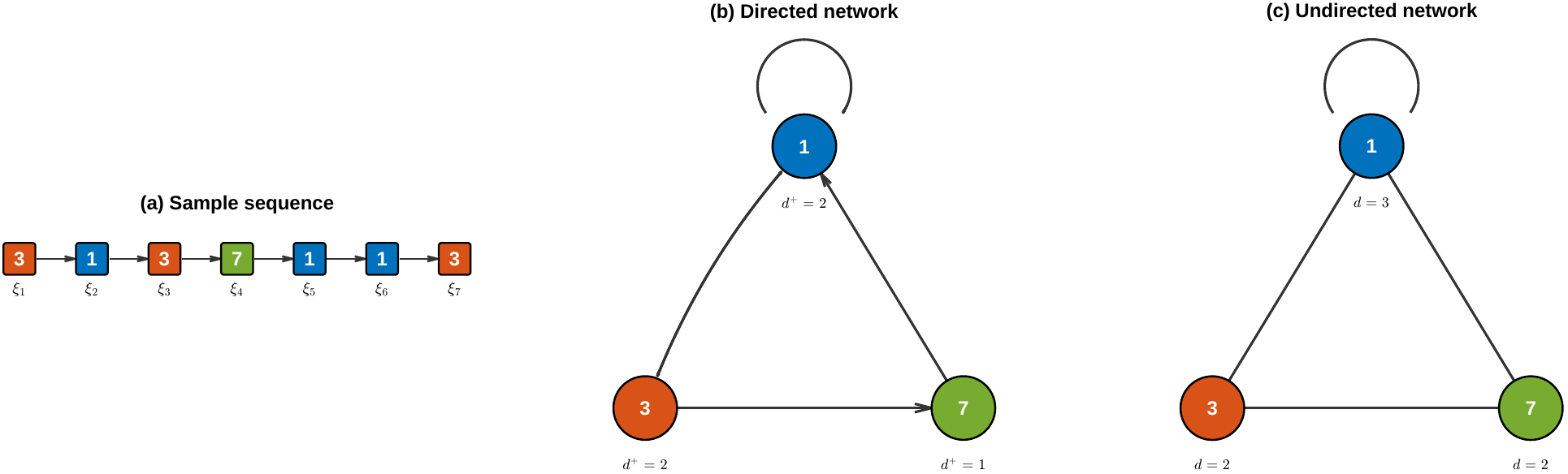}
\caption{Sequence-to-network construction for the sequence $(3,1,3,7,1,1,3)$. (a) Consecutive symbols generate transitions. (b) The directed graph retains one copy of each distinct transition, including the self-loop at vertex $1$; labels give distinct out-degrees. (c) Forgetting orientation gives the undirected graph; a self-loop contributes one neighbour.}
\label{fig:model}
\end{figure}

Throughout this paper, for a sequence $(a_k)$ we write $a_k\in\RV_\alpha$ if the step function $a_{\lceil x\rceil}$, $x\ge1$, is regularly varying at infinity with index $\alpha$ in the sense of Karamata \cite{BGT1987}. Equivalently, for every $t>0$,
\[
\frac{a_{\lceil tx\rceil}}{a_{\lceil x\rceil}} \longrightarrow t^\alpha
\qquad (x\to\infty).
\]

We assume that $\pi_k\in\RV_{-1/\gamma}$ for some $0<\gamma<1$. Equivalently, under the present monotonicity convention, define
\begin{equation}
\phi(x):=\frac{1}{\pi_{\lceil x\rceil}},\qquad x\ge1,
\end{equation}
so that
\begin{equation}
\phi\in\RV_{1/\gamma},\qquad \pi_n\phi(n)=1.
\label{eq:phi_def}
\end{equation}
Its generalized inverse counting function is
\begin{equation}
\alpha(t):=\#\{i:\pi_i\ge t^{-1}\}=\#\{i:\phi(i)\le t\}.
\end{equation}
By the monotone asymptotic inversion theorem \cite[Theorem~1.5.12]{BGT1987},
\begin{equation}
\alpha\in\RV_\gamma,\qquad
\alpha(\phi(n))\sim n,\qquad \phi(\alpha(t))\sim t.
\label{eq:inverse}
\end{equation}

Let $R_m$ also denote, when no confusion can arise, the number of occupied boxes after $m$ independent samples from $\pi$. Define the Sibuya weights
\begin{equation}
r_\ell(\gamma):=\frac{\gamma\,\Gamma(\ell-\gamma)}{\Gamma(1-\gamma)\Gamma(\ell+1)},\qquad \ell\ge1.
\label{eq:sibuya_mass}
\end{equation}
These weights form the Sibuya distribution \cite{Sibuya1979}; we will write
$r_\ell:=r_\ell(\gamma)$ for simplicity.
For finite $n$ and $k\ge1$, set
$\vec r_{n,k}:=\vec R_{n,k}/R_n$ and
$\bar r_{n,k}:=\bar R_{n,k}/R_n$.
Chen, Xie, and Ying proved that for every fixed $k\ge1$,
\begin{align}
\vec r_{n,k}&\longrightarrow
\vec r_k:=\sum_{\ell\ge1} r_\ell\, \Pnum(R_\ell=k),\label{eq:CXY_dir}\\
\bar r_{n,k}&\longrightarrow
\bar r_k:=\sum_{\ell\ge1} r_\ell\, \Pnum(R_{2\ell}=k)
\label{eq:CXY_undir}
\end{align}
almost surely \cite{CXY2022}. For any nonnegative summable sequence $(a_k)$, we use the tail convention
\begin{equation}
a_{k+}:=\sum_{\ell\ge k}a_\ell.
\label{eq:tail_notation}
\end{equation}
This applies in particular to the finite-sample tails
$\vec r_{n,k+},\bar r_{n,k+}$ and to the limiting tails
$\vec r_{k+},\bar r_{k+}$.

\section{Main results}\label{sec:results}

\subsection{Main asymptotic law}

\begin{theorem}[Degree asymptotics]\label{thm:main}
Assume that $\pi_k\in\RV_{-1/\gamma}$ for some $0<\gamma<1$. Then, as $k\to\infty$,
\begin{equation}
\vec r_{k+}\sim \pi_k^\gamma,\qquad
\bar r_{k+}\sim 2^\gamma\pi_k^\gamma,
\label{eq:tail_asym}
\end{equation}
and
\begin{equation}
\vec r_k\sim\frac{\pi_k^\gamma}{k},\qquad
\bar r_k\sim\frac{2^\gamma\pi_k^\gamma}{k}.
\label{eq:mass_asym}
\end{equation}
Consequently, both the directed and undirected limiting degree distributions are regularly varying with index $-2$; the only effect of forgetting edge orientation is a multiplicative factor $2^\gamma$ in the leading asymptotic term. Deleting self-loops does not alter these limiting distributions.
\end{theorem}

\begin{remark}
The universal index $-2$ admits a simple interpretation. The occupancy estimate $\pi_kN_k\to\Gamma(1-\gamma)^{-1/\gamma}$ and the Sibuya tail $Q_n\sim n^{-\gamma}/\Gamma(1-\gamma)$ give $\vec r_{k+}=\mathbb E[Q_{N_k}]\sim\pi_k^\gamma\in\RV_{-1}$. Tail regular variation alone does not determine a one-step decrement. The additional estimate $N_kM_k\sim\gamma k$ supplies the missing local information: the next-discovery gap is conditionally geometric with parameter $M_k$, and geometric smoothing gives $\vec r_k\sim r_{N_k}(\gamma)/M_k\sim Q_{N_k}/k$. For the undirected law, evaluating at approximately $N_k/2$ and doubling the effective success probability produces the factor $2^\gamma$. Hence the local exponent is $-2$ throughout the stated i.i.d. regular-variation class.
\end{remark}

To make the universal exponent explicit, write $L(k):=k^{1/\gamma}\pi_k$, which is slowly varying. Then
\begin{align}
\vec r_{k+}&\sim k^{-1}L(k)^\gamma,
&\bar r_{k+}&\sim2^\gamma k^{-1}L(k)^\gamma,\\
\vec r_k&\sim k^{-2}L(k)^\gamma,
&\bar r_k&\sim2^\gamma k^{-2}L(k)^\gamma.
\end{align}

\subsection{Exact Zipf benchmark}

For normalized Zipf frequencies
\begin{equation}
\pi_n=\frac{n^{-1/\gamma}}{\zeta(1/\gamma)},\qquad0<\gamma<1,
\label{eq:zipf}
\end{equation}
we have a pure power law with no additional slowly varying correction. Defining
\begin{equation}
C_\gamma:=\zeta(1/\gamma)^{-\gamma},
\label{eq:Cgamma}
\end{equation}
Theorem~\ref{thm:main} gives
\begin{align}
\vec r_{k+}&\sim C_\gamma k^{-1},
&\bar r_{k+}&\sim2^\gamma C_\gamma k^{-1},\\
\vec r_k&\sim C_\gamma k^{-2},
&\bar r_k&\sim2^\gamma C_\gamma k^{-2}.
\end{align}

\section{Discovery representation}\label{sec:representation}

Recall that the Sibuya tail is
\begin{equation}
Q_n:=r_{n+}=\sum_{\ell\ge n}r_\ell
=\frac{\Gamma(n-\gamma)}{\Gamma(1-\gamma)\Gamma(n)}
\sim\frac{n^{-\gamma}}{\Gamma(1-\gamma)}.
\label{eq:Q_summary}
\end{equation}
Its one-step decrement satisfies the exact recurrence
\begin{equation}
Q_n-Q_{n+1}=r_n(\gamma)=\frac{\gamma}{n}Q_n.
\label{eq:Q_rec}
\end{equation}
For $k\ge1$, define the discovery index, inter-discovery gap, and residual unseen mass by
\begin{align}
N_k&:=\inf\{n:R_n=k\},&G_k&:=N_{k+1}-N_k,\label{eq:N_G_def}\\
M_k&:=\sum_{i:\,i\notin\{\xi_1,\ldots,\xi_{N_k}\}}\pi_i.
\label{eq:M_def}
\end{align}
Thus $N_k$ is the number of samples required to observe $k$ distinct symbols. Let $\mathcal F_{N_k}$ denote the history up to time $N_k$. Then the next-discovery gap is conditionally geometric:
\begin{equation}
\Pnum(G_k>j\mid\mathcal F_{N_k})=(1-M_k)^j,
\qquad j\ge0.
\label{eq:geom_conditional}
\end{equation}

Since $R_\ell\ge k$ if and only if $\ell\ge N_k$, the directed degree tail satisfies
\begin{equation}
\vec r_{k+}
=\sum_{\ell\ge1} r_\ell\, \Pnum(R_\ell\ge k)
=\Enum\left[Q_{N_k}\right].
\label{eq:dir_tail_expect}
\end{equation}
Similarly, since $R_{2\ell}\ge k$ if and only if $2\ell\ge N_k$, the undirected degree tail satisfies
\begin{equation}
\bar r_{k+}
=\sum_{\ell\ge1} r_\ell\, \Pnum(R_{2\ell}\ge k)
=\Enum\left[Q_{\lceil N_k/2\rceil}\right].
\label{eq:undir_tail_expect}
\end{equation}

For the directed local mass, Eq.~\eqref{eq:geom_conditional} and the identity in Ref.~\cite{CXY2022} give
\begin{equation}
\vec r_k=\Enum H_{k,1},\qquad
H_{k,1}:=\sum_{j\ge0}r_{N_k+j}(\gamma)(1-M_k)^j.
\label{eq:H1}
\end{equation}
For the undirected local mass, put
\begin{equation}
a_k:=\left\lceil\frac{N_k}{2}\right\rceil,
\qquad \epsilon_k:=2a_k-N_k\in\{0,1\}.
\label{eq:a_epsilon}
\end{equation}
Since $R_{2\ell}=k$ if and only if $N_k\le2\ell<N_{k+1}$, conditioning as above yields
\begin{align}
\bar r_k&=\Enum H_{k,2},\notag\\
H_{k,2}&:=(1-M_k)^{\epsilon_k}
\sum_{j\ge0}r_{a_k+j}(\gamma)(1-M_k)^{2j}.
\label{eq:H2}
\end{align}

\section{Proof of Theorem~\ref{thm:main}}\label{sec:proof}

\subsection{Occupancy estimates and degree tails}

We use two fundamental consequences of infinite-occupancy theory.

\begin{proposition}[\cite{GHP2007}, Propositions~23--25]
\label{prop:piN}
Under the assumption $\pi_k\in\RV_{-1/\gamma}$,
\begin{align}
\pi_k N_k \longrightarrow \Gamma(1-\gamma)^{-1/\gamma}
&\quad\text{a.s.},\label{eq:piN_limit}\\
\frac{N_kM_k}{k}\longrightarrow\gamma
&\quad\text{a.s.}
\label{eq:NM_limit}
\end{align}
\end{proposition}

\begin{proof}
In the notation of Gnedin, Hansen, and Pitman \cite{GHP2007}, the probabilities $p_j$ there are our $\pi_j$, the number of occupied boxes $K_n$ there is our $R_n$, and the discovery time $N_k$ is the same. Their regular-variation parameter $\alpha\in(0,1)$ is our $\gamma$. Under this correspondence, Proposition~24 gives
\[
N_k \sim \phi\!\left(\frac{k}{\Gamma(1-\gamma)}\right)
\quad\text{a.s.}
\]
Since $\phi\in\RV_{1/\gamma}$ and $\pi_k=1/\phi(k)$, we have
$\phi(k/\Gamma(1-\gamma))\sim \phi(k)\Gamma(1-\gamma)^{-1/\gamma}
= \pi_k^{-1}\Gamma(1-\gamma)^{-1/\gamma}$.
Therefore \eqref{eq:piN_limit} follows. Their residual mass after $k$ discoveries is our $M_k$; Proposition~25, evaluated at the $k$th discovery time, gives \eqref{eq:NM_limit}.
\end{proof}

In addition, the following exponential lower-tail estimate for $N_k$ will be needed; its proof is included in Appendix~\ref{app:tail_bound}. There exist constants $a_0,c_0>0$ such that for all sufficiently large $k$,
\begin{equation}
\Pnum(N_k<\phi(a_0 k))\le e^{-c_0 k}.
\label{eq:tail_bound}
\end{equation}

By Proposition~\ref{prop:piN}, with $Y_k:=(\pi_k N_k)^{-\gamma}$, we have
\begin{equation}
Y_k \longrightarrow \Gamma(1-\gamma)
\quad\text{a.s.}
\label{eq:Y_conv}
\end{equation}

We claim that $\{Y_k\}$ is uniformly integrable. For any $\delta>0$, using the layer-cake representation and noting that $N_k\ge1$,
\begin{align}
\Enum[Y_k^{1+\delta}]
&= \frac{\gamma(1+\delta)}{\pi_k^{\gamma(1+\delta)}}
\int_1^\infty s^{-\gamma(1+\delta)-1} \Pnum(N_k < s) \, ds.
\end{align}
Let $u_k:=\phi(a_0 k)$. Splitting the integral at $u_k$ gives
\begin{align}
\Enum[Y_k^{1+\delta}]
&\le \frac{e^{-c_0 k} + u_k^{-\gamma(1+\delta)}}{\pi_k^{\gamma(1+\delta)}}.
\end{align}
Since $u_k=\phi(a_0 k)\sim a_0^{1/\gamma}/\pi_k$, we have
$u_k^{-\gamma(1+\delta)}\sim a_0^{-(1+\delta)} \pi_k^{\gamma(1+\delta)}$.
Thus there exists a constant $c_1>0$ such that
\begin{equation}
\Enum[Y_k^{1+\delta}]
\le \frac{e^{-c_0 k}}{\pi_k^{\gamma(1+\delta)}} + c_1 a_0^{-(1+\delta)}.
\end{equation}
Since $\pi_k\in\RV_{-1/\gamma}$, for any $\varepsilon>0$ there exists $C>0$ such that $\pi_k\ge C k^{-1/\gamma-\varepsilon}$. Hence
\[
\frac{e^{-c_0 k}}{\pi_k^{\gamma(1+\delta)}} \le C k^{(1+\delta)(1+\gamma\varepsilon)} e^{-c_0 k} \longrightarrow 0,
\]
so $\sup_k \Enum[Y_k^{1+\delta}]<\infty$. Uniform integrability follows.

By the Vitali convergence theorem, \eqref{eq:Y_conv} and uniform integrability imply
\begin{equation}
\Enum[Y_k] \longrightarrow \Gamma(1-\gamma),
\end{equation}
or equivalently,
\begin{equation}
\Enum[N_k^{-\gamma}] \sim \Gamma(1-\gamma) \pi_k^\gamma.
\label{eq:EN_limit}
\end{equation}

Because $N_k\ge k$ and the ratio in \eqref{eq:Q_summary} converges uniformly over all integer arguments at least $k$, Eqs.~\eqref{eq:dir_tail_expect} and \eqref{eq:EN_limit} give
\begin{equation}
\vec r_{k+}
= \Enum[Q_{N_k}]
\sim \frac{1}{\Gamma(1-\gamma)} \Enum[N_k^{-\gamma}]
\sim \pi_k^\gamma.
\label{eq:dir_tail_result}
\end{equation}
Applying the same uniform argument at half the discovery index yields
\begin{equation}
\bar r_{k+}
= \Enum[Q_{\lceil N_k/2\rceil}]
\sim \frac{2^\gamma}{\Gamma(1-\gamma)} \Enum[N_k^{-\gamma}]
\sim 2^\gamma \pi_k^\gamma.
\label{eq:undir_tail_result}
\end{equation}

\subsection{Geometric smoothing}

The following lemma extracts the local scale from the conditional geometric sums in Eqs.~\eqref{eq:H1} and \eqref{eq:H2}.

\begin{lemma}[Geometric smoothing]\label{lem:smooth}
Let $(b_n)$ be a positive sequence with $b_n\in\RV_\beta$. If $n_k$ are positive integers and $m_k\in(0,1]$ satisfy $n_km_k\to\infty$, then
\begin{equation}
\sum_{j\ge0}b_{n_k+j}(1-m_k)^j
\sim\frac{b_{n_k}}{m_k}.
\label{eq:smooth}
\end{equation}
\end{lemma}

\begin{proof}
Let $J_k$ be geometric on $\{0,1,\ldots\}$ with $\Pnum(J_k=j)=m_k(1-m_k)^j$. The ratio of the two sides of \eqref{eq:smooth} is $\Enum[b_{n_k+J_k}/b_{n_k}]$. Since
\begin{equation}
\Pnum(J_k>\delta n_k)\le
\exp\{-m_k\lfloor\delta n_k\rfloor\}\longrightarrow0,
\end{equation}
we have $J_k/n_k\to0$ in probability. The uniform convergence theorem for regularly varying sequences gives $b_{n_k+J_k}/b_{n_k}\to1$ in probability. It remains to justify convergence of expectations. For any $p>1$ and sufficiently small $\varepsilon>0$, Potter's bound gives
\begin{equation}
\frac{b_{n_k+j}}{b_{n_k}}
\le C_\varepsilon\left(1+\frac{j}{n_k}\right)^a,
\qquad a:=(\beta+\varepsilon)_+.
\end{equation}
Choose an integer $d>ap$. Since $\Enum J_k^d\le C_dm_k^{-d}$,
\begin{equation}
\sup_{k\ \mathrm{large}}
\Enum\left(\frac{b_{n_k+J_k}}{b_{n_k}}\right)^p
\le C\left[1+(n_km_k)^{-d}\right]<\infty.
\end{equation}
Thus the ratios are uniformly integrable, and their expectations converge to $1$.
\end{proof}

\subsection{Local degree masses}

The Sibuya masses satisfy $r_n(\gamma)\in\RV_{-1-\gamma}$. On the probability-one event where Proposition~\ref{prop:piN} holds, $N_kM_k\to\infty$, so Lemma~\ref{lem:smooth} applied to \eqref{eq:H1} gives
\begin{equation}
H_{k,1}\sim\frac{r_{N_k}(\gamma)}{M_k}.
\label{eq:H1_asym}
\end{equation}
For the undirected case, set $m_k^{(2)}:=1-(1-M_k)^2\sim2M_k$. Since $a_k\sim N_k/2$ and $a_km_k^{(2)}\sim N_kM_k\to\infty$, Eqs.~\eqref{eq:H2} and \eqref{eq:smooth} give
\begin{equation}
H_{k,2}\sim
(1-M_k)^{\epsilon_k}\frac{r_{a_k}(\gamma)}{m_k^{(2)}}
\sim2^\gamma\frac{r_{N_k}(\gamma)}{M_k}.
\label{eq:H2_asym}
\end{equation}
Using the recurrence \eqref{eq:Q_rec}, Proposition~\ref{prop:piN}, and \eqref{eq:NM_limit}, both cases can be written as
\begin{equation}
H_{k,c}
\sim c^\gamma\frac{r_{N_k}(\gamma)}{M_k}
=c^\gamma\frac{\gamma Q_{N_k}}{N_kM_k}
\sim\frac{c^\gamma\pi_k^\gamma}{k},
\qquad c=1,2,
\label{eq:H_pathwise}
\end{equation}
almost surely.

It remains to pass from \eqref{eq:H_pathwise} to expectations. Put
\begin{align}
s_{k,c}&:=\frac{c^\gamma\pi_k^\gamma}{k},&
Y_{k,c}&:=\frac{H_{k,c}}{s_{k,c}},\notag\\
B_k&:=\{N_k\ge\phi(a_0k)\}.&&
\end{align}
Exactly $k$ symbols have been observed by time $N_k$, so the ordering of $(\pi_i)$ and Karamata's theorem imply
\begin{equation}
M_k\ge\sum_{i>k}\pi_i
\sim\frac{\gamma}{1-\gamma}\,k\pi_k
=\frac{\gamma}{1-\gamma}\frac{k}{\phi(k)}.
\label{eq:M_lower}
\end{equation}
Monotonicity of the Sibuya masses and Eqs.~\eqref{eq:H1}--\eqref{eq:H2} yield, for $c=1,2$,
\begin{equation}
H_{k,c}\le\frac{r_{\lceil N_k/c\rceil}(\gamma)}{M_k}.
\end{equation}
Consequently, regular variation and \eqref{eq:M_lower} give the deterministic bound
\begin{equation}
Y_{k,c}\mathbf1_{B_k}
\le C\left(\frac{\phi(k)}{N_k}\right)^{1+\gamma}
\mathbf1_{B_k}\le C_B.
\end{equation}
On $B_k^c$, $H_{k,c}\le1$ and Potter's bound gives $Y_{k,c}\le C_\delta k^{2+\delta}$. Therefore the exponential estimate \eqref{eq:tail_bound} implies
\begin{equation}
\Enum[Y_{k,c};B_k^c]
\le C_\delta k^{2+\delta}e^{-c_0k}\longrightarrow0.
\end{equation}
Thus $(Y_{k,c})$ is uniformly integrable. Combining this with \eqref{eq:H_pathwise} yields
\begin{equation}
\vec r_k=\Enum H_{k,1}\sim\frac{\pi_k^\gamma}{k},
\qquad
\bar r_k=\Enum H_{k,2}\sim\frac{2^\gamma\pi_k^\gamma}{k}.
\label{eq:local_result}
\end{equation}

This completes the proof of Theorem~\ref{thm:main}. The loop-deleted assertions follow from Proposition~\ref{prop:loop_deletion} in Appendix~\ref{app:loop_deletion}.

\section{Finite-sample illustration}\label{sec:numerics}

We now illustrate the directed tail asymptotic in the normalized Zipf family \eqref{eq:zipf}. By the convention in \eqref{eq:tail_notation},
\begin{equation}
\vec r_{n,k+}
=\sum_{\ell\ge k}\vec r_{n,\ell}
=\frac{1}{R_n}\#\{v:D_n^+(v)\ge k\}.
\label{eq:finite_directed_tail}
\end{equation}
For each fixed $k$, the strong laws in \cite{CXY2022} imply $\vec r_{n,k+}\to\vec r_{k+}$ almost surely. Consequently, the Zipf case satisfies the iterated limit
\begin{equation}
\lim_{k\to\infty}\lim_{n\to\infty}
\frac{k\vec r_{n,k+}}{C_\gamma}=1
\qquad\text{a.s.}
\label{eq:iterated_tail}
\end{equation}
No simultaneous limit with $k=k(n)$ is claimed.

Figure~\ref{fig:tail} displays \eqref{eq:iterated_tail} as a linear relation by plotting $1/\vec r_{n,k+}$ against $k$. The dashed line is the theoretical prediction $k/C_\gamma$, with $C_\gamma$ fixed by \eqref{eq:Cgamma}; no parameter is fitted to the simulations.

\begin{figure}[t]
\centering
\includegraphics[width=\linewidth]{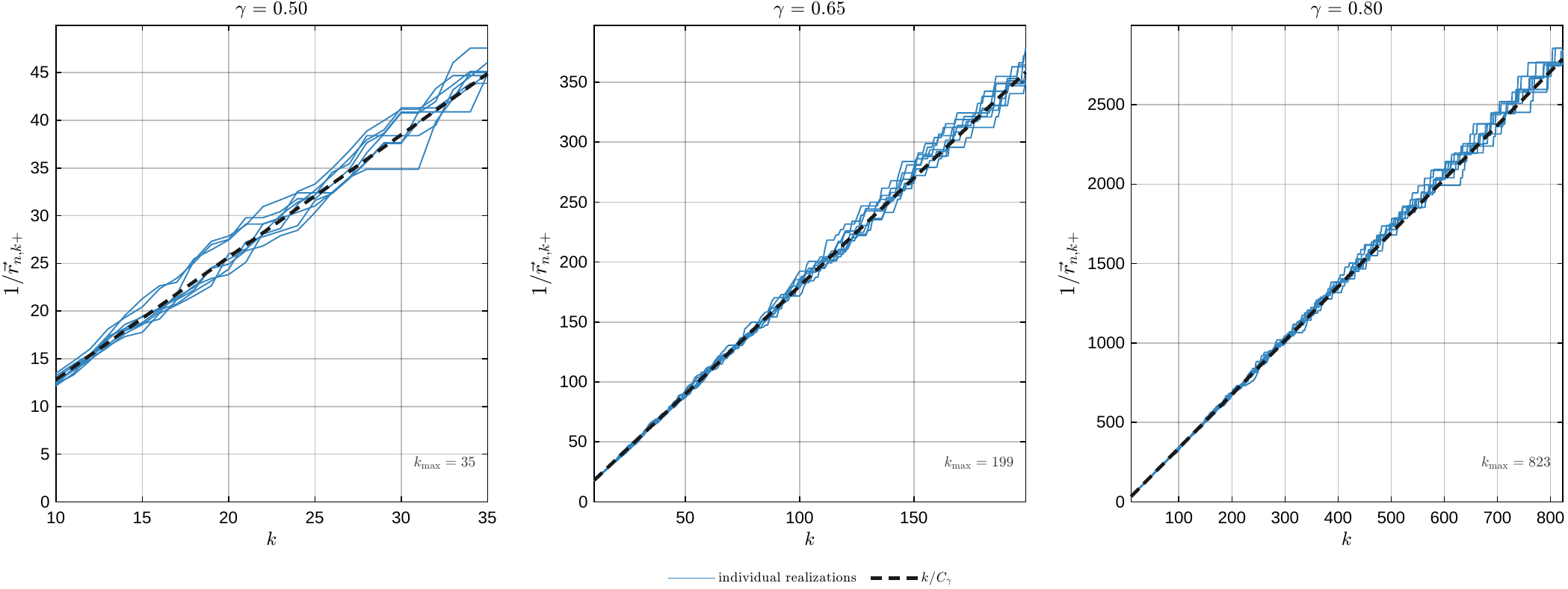}
\caption{Reciprocal finite-sample directed degree tails for normalized Zipf frequencies with $n=10^6$ and $\gamma=0.50,0.65,0.80$ (left to right). Each panel shows eight independent realisations separately (thin blue curves); no averaging is used. The dashed line is the theoretical benchmark $k/C_\gamma$. Both axes are linear. The displayed interval begins at $k=10$ and ends at the largest threshold for which every realisation has at least $30$ vertices in the directed degree tail.}
\label{fig:tail}
\end{figure}

For each value of $\gamma$ we generated eight prespecified independent sequences of length $n=10^6$. The mean numbers of observed vertices were $1.403\times10^3$, $1.123\times10^4$, and $8.556\times10^4$ for $\gamma=0.50,0.65,0.80$, respectively. The corresponding mean numbers of distinct directed edges were $8.191\times10^3$, $6.166\times10^4$, and $3.282\times10^5$, giving mean directed out-degrees $5.841$, $5.492$, and $3.836$.

The common lower endpoint $k_{\min}=10$ was fixed before inspecting the reciprocal-tail plot. To avoid magnifying extremely sparse counts, the upper endpoint is the largest $k$ for which every realisation contains at least $30$ vertices of out-degree at least $k$. This rule gives $k_{\max}=35$, $199$, and $823$. The spread grows near the sparsest retained tail, as expected for a reciprocal count, but the individual curves remain close to the predicted line over the displayed intervals.

The simulations used a rejection sampler for the infinite Zipf law. Consecutive samples define directed edges and repeated edges are removed. To avoid loss of integer precision, ranks above $K=10^{12}$ were assigned fresh labels. This replacement can alter the graph only if two such samples had the same original rank; a union bound is at most $1.58\times10^{-8}$ for the worst simulated parameter combination.

\section{Discussion and outlook}\label{sec:discussion}

The exponent-two law results from a combination of two structures. Regular variation controls the discovery scale through $\pi_kN_k$, while the residual unseen mass determines the local gap scale through $N_kM_k$. The first relation is sufficient for the degree tails; the second is what makes the local masses accessible. This distinction is mathematically essential because a tail asymptotic cannot in general be differentiated term by term. Conditional geometric smoothing provides the additional local regularity in a form adapted to the discovery process.

The model also separates the effects of sampling heterogeneity, edge orientation and repeated transitions. The activity exponent $\gamma$ controls the rate at which new vertices appear, but it cancels from the local degree exponent. Forgetting orientation multiplies the leading term by $2^\gamma$. Suppressing repeated transitions is essential: in an adjacency multigraph, multiplicities continue to contribute to degree \cite{BedogneRodgers2008}. By contrast, deleting self-loops is asymptotically immaterial. The number of loop-carrying vertices is $O(\alpha(\sqrt n))$ almost surely, compared with $R_n\asymp\alpha(n)$ observed vertices, so the affected proportion is $n^{-\gamma/2+o(1)}$.

The comparison with Janson's rank-one graph \cite{Janson2018} suggests that exponent-two tails are stable across more than one endpoint-sampling construction. The mathematical inputs differ, however: Janson's edges are independent, whereas consecutive range-renewal edges overlap. The present argument identifies the discovery-time estimates that recover sharp tails and local masses in the overlapping setting.

Two natural extensions remain open. First, a joint regime $k=k(n)$ would quantify the degree window in which the limiting formula approximates a finite graph and would explain the increasing fluctuations in the far tail of Fig.~\ref{fig:tail}. Second, for a dependent symbol sequence with marginal law $\pi$, temporal dependence changes both discovery gaps and transition repetitions. Determining conditions under which the exponent $2$ and the orientation ratio $2^\gamma$ persist would connect the present occupancy argument to sequence-generated graphs from correlated data.

\section*{Data availability statement}

All numerical results were generated in MATLAB by direct Monte Carlo simulation of \eqref{eq:zipf}. The seeded sampler and network-construction code, validation scripts, fixed seeds, and processed data underlying Figs.~\ref{fig:model} and \ref{fig:tail} are archived at Zenodo: \href{https://doi.org/10.5281/zenodo.22052914}{doi:10.5281/zenodo.22052914}.

\appendix

\section{Proof of the exponential lower-tail estimate}\label{app:tail_bound}

We provide a proof of the exponential lower-tail estimate used in the uniform integrability argument. This is a standard consequence of the infinite-occupancy theory; we include the details for completeness.

Let $a_0>0$ be chosen so small that $2^\gamma \Gamma(1-\gamma) a_0 < e^{-2}$. Set $n_k:=\lceil\phi(a_0 k)\rceil$ and $t_k:=2n_k$.

Let $\{\mathcal N(t):t\ge0\}$ be an independent unit-rate Poisson process and define the Poissonized range
\[
R(t):=\#\{\xi_1,\ldots,\xi_{\mathcal N(t)}\}.
\]
Then $R(t)$ is a sum of independent Bernoulli variables with mean
\[
\mu(t):=\sum_i(1-e^{-t\pi_i})\sim \Gamma(1-\gamma)\alpha(t).
\]
Since $\alpha(n_k)\sim a_0k$, regular variation gives
\[
\alpha(t_k)\sim2^\gamma a_0k,
\qquad
\mu(t_k)\sim2^\gamma\Gamma(1-\gamma)a_0k.
\]
Hence $\mu(t_k)\le e^{-2}k$ for all sufficiently large $k$ by the choice of $a_0$.

Coupling the fixed-size and Poissonized experiments gives
\[
\Pnum(N_k<n_k)
\le \Pnum(R(t_k)\ge k)+\Pnum(\mathrm{Poi}(t_k)<n_k).
\]
By Chernoff's inequality,
\[
\Pnum(R(t_k)\ge k)
\le \left(\frac{e\mu(t_k)}{k}\right)^k
\le e^{-k},
\]
and
\[
\Pnum(\mathrm{Poi}(2n_k)<n_k)\le e^{-n_k/4}.
\]
Since $\phi\in\RV_{1/\gamma}$ with $1/\gamma>1$, $n_k/k\to\infty$, so both terms are exponentially small. Thus \eqref{eq:tail_bound} follows for a suitable $c_0>0$.

\section{Robustness to self-loop deletion}\label{app:loop_deletion}

For $x\in V_n$, define the loop-deleted directed and undirected degrees by
\begin{align}
D_n^{+,\circ}(x)&:=\#\bigl(V_n^+(x)\setminus\{x\}\bigr),\\
\bar D_n^\circ(x)&:=\#\bigl(\bar V_n(x)\setminus\{x\}\bigr).
\end{align}
Let $\vec R_{n,k}^{\,\circ}$ and $\bar R_{n,k}^{\,\circ}$ be the corresponding degree counts, normalized by $R_n$ to obtain $\vec r_{n,k}^{\,\circ}$ and $\bar r_{n,k}^{\,\circ}$.

\begin{proposition}[Deletion of self-loops]\label{prop:loop_deletion}
Under the assumptions of Theorem~\ref{thm:main}, for every fixed $k\ge1$,
\begin{equation}
\vec r_{n,k}^{\,\circ}\longrightarrow\vec r_k,
\qquad
\bar r_{n,k}^{\,\circ}\longrightarrow\bar r_k
\quad\text{a.s.}
\end{equation}
The loop-deleted finite-sample tails converge to $\vec r_{k+}$ and $\bar r_{k+}$ as well. Hence the limiting degree distributions are unchanged by the deletion of self-loops.
\end{proposition}

\begin{proof}
Let $\Lambda_n:=\#\{x:W_n(x,x)\ge1\}$ be the number of vertices carrying a self-loop. Split the adjacent pairs into two families: $(\xi_1,\xi_2),(\xi_3,\xi_4),\ldots$ and $(\xi_2,\xi_3),(\xi_4,\xi_5),\ldots$. Within either family the pairs are independent. Encode a pair by the label $i$ if both entries equal $i$, and by $0$ otherwise. The resulting i.i.d.\ variables have frequencies $p_i=\pi_i^2$ and $p_0=1-\sum_i\pi_i^2$. The counting function of this frequency sequence is
\[
\alpha_\circ(t)=\alpha(\sqrt t)+O(1)\in\RV_{\gamma/2}.
\]
The standard strong law for infinite occupancy \cite{GHP2007,Karlin1968}, applied separately to the two families, shows that each family contains $O(\alpha(\sqrt n))$ distinct nonzero labels a.s. Hence $\Lambda_n=O(\alpha(\sqrt n))$ a.s. The same occupancy strong law gives $R_n\sim\Gamma(1-\gamma)\alpha(n)$ a.s. Regular variation then yields
\[
\frac{\Lambda_n}{R_n}
=O\!\left(\frac{\alpha(\sqrt n)}{\alpha(n)}\right)
=O\!\left(n^{-\gamma/2+o(1)}\right)\longrightarrow0.
\]
Deleting self-loops changes the degree only at these $\Lambda_n$ vertices, and there by exactly one. Consequently, each degree count differs from its retained-loop counterpart by at most $\Lambda_n$. After division by $R_n$, the difference tends to zero a.s. Combining this with the retained-loop fixed-$k$ strong laws \eqref{eq:CXY_dir} and \eqref{eq:CXY_undir} proves the proposition.
\end{proof}

\end{document}